\documentclass{amsart}
\usepackage{setspace}
\usepackage{a4}
\usepackage{amssymb,amsmath,amsthm,latexsym}
\usepackage{amsfonts}
\usepackage{amsfonts}
\usepackage{graphicx}
\usepackage{textcomp}
\usepackage{cite}
\usepackage{enumerate}
\usepackage[mathscr]{euscript}
\usepackage{mathtools}
\newtheorem{theorem}{Theorem}[section]

\newtheorem{corollary}[theorem] {Corollary}
\newtheorem{definition}[theorem]{Definition}

\title{This is the title}
\usepackage{amssymb}
\usepackage{amssymb}
\usepackage{amssymb}
\usepackage{amssymb}
\usepackage{amsmath}
\usepackage{tikz}
\usepackage{hyperref}
\usepackage{enumerate}
\usepackage{mathtools}
\usepackage{amsmath}
\usepackage{tikz}
\usepackage{amssymb}
\usepackage{amsmath}
\usepackage{tikz}
\usepackage{hyperref}
\begin{document}
\begin{center}
{\bf{DILATION THEOREM FOR  p-APPROXIMATE SCHAUDER FRAMES FOR SEPARABLE BANACH SPACES}}\\
K. MAHESH KRISHNA AND P. SAM JOHNSON  \\
Department of Mathematical and Computational Sciences\\ 
National Institute of Technology Karnataka (NITK), Surathkal\\
Mangaluru 575 025, India  \\
Emails: kmaheshak@gmail.com,  sam@nitk.edu.in\\

Date: \today
\end{center}

\hrule
\vspace{0.5cm}
\textbf{Abstract}: Famous Naimark-Han-Larson dilation theorem for frames in Hilbert spaces states that every frame for a separable Hilbert space $\mathcal{H}$ is image of a Riesz basis under an orthogonal projection  from a separable Hilbert space $\mathcal{H}_1$ which contains  $\mathcal{H}$ isometrically. In this paper, we derive dilation result for p-approximate Schauder frames for separable Banach spaces. Our result contains Naimark-Han-Larson dilation theorem as a particular case. 




\section{Introduction}
Let  $\mathbb{K}=\mathbb{R}$ or $\mathbb{C}$ and  $\mathcal{H}$ be a separable Hilbert space over $\mathbb{K}$. Folklore definitions of Riesz basis for $\mathcal{H}$ and frame for  $\mathcal{H}$ are the following. 
\begin{definition}\cite{BARI1, BARI2}
	A sequence $\{\tau_n\}_n$ in  $\mathcal{H}$ is said to be a Riesz basis for 
	 $\mathcal{H}$ if there exists an orthonormal basis $\{\omega_n\}_n$ for  $\mathcal{H}$ and a bounded invertible linear operator $T: \mathcal{H}\to \mathcal{H}$ such that 
	\begin{align*}
	T\omega_n=\tau_n, \quad \forall n \in \mathbb{N}.
	\end{align*}
\end{definition}
 \begin{definition}\cite{DUFFIN}
A sequence $\{\tau_n\}_n$ in  $\mathcal{H}$ is said to be a
frame for $\mathcal{H}$ if there exist $a,b>0$ such that 
\begin{align*}
a\|h\|^2 \leq \sum_{n=1}^\infty |\langle h, \tau_n\rangle|^2\leq b\|h\|^2, \quad \forall h \in \mathcal{H}.
\end{align*}
\end{definition}
Dilation theory usually tries to extend  operator on Hilbert space to larger Hilbert space which are easier to handle as well as well-understood and study the original operator  as a slice of it \cite{LEVYSHALIT, ARVESON, NAGY}. As long as frame theory for Hilbert spaces is considered,  following theorem is known as   Naimark-Han-Larson dilation theorem. This was proved independently by Han and Larson in 2000 \cite{HANMEMOIRS} and by Kashin and Kukilova in 2002 \cite{KASHINKULIKOVA}. We refer the reader to  \cite{CZAJA} for the history of  this theorem.
\begin{theorem}(Naimark-Han-Larson dilation theorem) \cite{HANMEMOIRS, KASHINKULIKOVA}\label{NHLDILATIONTHEOREM}
 Let $\{\tau_n\}_n$ be  a  frame  for  $ \mathcal{H}$.  Then there exist a Hilbert space $ \mathcal{H}_1 $ which contains $ \mathcal{H}$ isometrically and  a Riesz basis $\{\omega_n\}_n$ for  $ \mathcal{H}_1$ such that 
 \begin{align*}
 \tau_n=P\omega_n, \quad\forall n \in \mathbb{N},
 \end{align*}
  where $P$ is the orthogonal projection from $\mathcal{H}_1$ onto $\mathcal{H}$. 	
\end{theorem}
Reason for names Han and Larson in Theorem \ref{NHLDILATIONTHEOREM} is clearly evident whereas that of Naimark is that Theorem \ref{NHLDILATIONTHEOREM} is a particular case of famous Naimark dilation theorem (see the introduction of the paper \cite{CZAJA}).  To the best of our knowledge,  proofs of Theorem \ref{NHLDILATIONTHEOREM} can be found in \cite{CASAZZAHANLARSON}, \cite{HANMEMOIRS}, \cite{SKOPINA} and \cite{CZAJA}. By the way, proof of dilation theorem in the finite dimensional case can  be found in \cite{KORNELSON} and \cite{CASAZZABOOK}.\\
In this paper, we derive  dilation theorem for p-approximate Schauder frames for separable Banach spaces (Theorem \ref{DILATIONTHEOREMPASF}). Theorem \ref{NHLDILATIONTHEOREM} then becomes a particular case of Theorem \ref{DILATIONTHEOREMPASF}.

\section{Dilation theorem for p-approximate Schauder frames}
Following theorem is the fundamental result in frame theory for Hilbert spaces which motivates the definition of frames for Banach spaces.
\begin{theorem}\cite{DUFFIN, HANMEMOIRS}\label{FFF}
	Let $\{\tau_n\}_n$ be a frame for $\mathcal{H}$. Then
	\begin{enumerate}[\upshape(i)]
		\item The map
		$
		S_\tau :\mathcal{H} \ni h \mapsto \sum_{n=1}^\infty \langle h, \tau_n\rangle\tau_n\in
		\mathcal{H}
		$
		is  a well-defined  bounded linear, positive and invertible operator. Further, 
		\begin{align}\label{REP}
		\text{(general Fourier expansion)} \quad h=\sum_{n=1}^\infty \langle h, S_\tau^{-1}\tau_n\rangle \tau_n=\sum_{n=1}^\infty
		\langle h, \tau_n\rangle S_\tau^{-1}\tau_n, \quad \forall h \in
		\mathcal{H}.
		\end{align}
	\item The map 
		$
		\theta_\tau: \mathcal{H} \ni h \mapsto \{\langle h, \tau_n\rangle \}_n \in \ell^2(\mathbb{N})
		$
		is  a well-defined  bounded linear, injective operator.
		\item   Adjoint of $\theta_\tau$ 
	 is given by 
		$
		\theta_\tau^*: \ell^2(\mathbb{N}) \ni \{a_n \}_n \mapsto \sum_{n=1}^\infty a_n\tau_n \in \mathcal{H}
		$
		which is surjective.
		\item  $S_\tau=\theta_\tau^*\theta_\tau.$
		\item $ P_\tau \coloneqq \theta_\tau S_\tau^{-1} \theta_\tau^*:\ell^2(\mathbb{N}) \to \ell^2(\mathbb{N})$ is an orthogonal  projection onto $ \theta_\tau(\mathcal{H})$.
	\end{enumerate}
\end{theorem}
 Let $\mathcal{X}$ be a separable Banach space and $\mathcal{X}^*$ be its dual.   General Fourier expansion in Equation (\ref{REP}) allows to define   the notion of Schauder frame for $\mathcal{X}$. 
\begin{definition}\cite{CASAZZA}\label{FRAMING}
	Let $\{\tau_n\}_n$ be a sequence in  $\mathcal{X}$ and 	$\{f_n\}_n$ be a sequence in  $\mathcal{X}^*.$ The pair $ (\{f_n \}_{n}, \{\tau_n \}_{n}) $ is said to be a Schauder frame for $\mathcal{X}$ if 
	\begin{align*}
	x=\sum_{n=1}^\infty
	f_n(x)\tau_n, \quad \forall x \in
	\mathcal{X}.
	\end{align*}
\end{definition} 
Notion of Schauder frame has a very natural generalization which is stated as below.
\begin{definition}\cite{FREEMANODELL, THOMAS}\label{ASFDEF}
	Let $\{\tau_n\}_n$ be a sequence in  $\mathcal{X}$ and 	$\{f_n\}_n$ be a sequence in  $\mathcal{X}^*.$ The pair $ (\{f_n \}_{n}, \{\tau_n \}_{n}) $ is said to be an approximate Schauder frame (ASF) for $\mathcal{X}$ if 
	\begin{align*}
	S_{f, \tau}:\mathcal{X}\ni x \mapsto S_{f, \tau}x\coloneqq \sum_{n=1}^\infty
	f_n(x)\tau_n \in
	\mathcal{X}
	\end{align*}
	is a well-defined bounded linear, invertible operator.
\end{definition} 
Recently, a particular case of Definition \ref{ASFDEF} was studied by same authors of this paper by defining p-approximate Schauder frames (p-ASFs).
\begin{definition}\cite{MAHESHJOHNSON}\label{PASFDEF}
	An ASF $ (\{f_n \}_{n}, \{\tau_n \}_{n}) $  for $\mathcal{X}$	is said to be p-ASF, $p \in [1, \infty)$ if both the maps 
	\begin{align*}
	& \theta_f: \mathcal{X}\ni x \mapsto \theta_f x\coloneqq \{f_n(x)\}_n \in \ell^p(\mathbb{N}), \\
	&\theta_\tau : \ell^p(\mathbb{N}) \ni \{a_n\}_n \mapsto \theta_\tau \{a_n\}_n\coloneqq \sum_{n=1}^\infty a_n\tau_n \in \mathcal{X}
	\end{align*}
	are well-defined bounded linear operators. 
\end{definition}
Advantage of p-ASF is that it gives a result similar to that of Theoerm \ref{FFF}. 
\begin{theorem}\cite{MAHESHJOHNSON}\label{OURS}
	Let $ (\{f_n \}_{n}, \{\tau_n \}_{n}) $ be a p-ASF for $\mathcal{X}$.  Then
	\begin{enumerate}[\upshape(i)]
		\item We have 
		\begin{align*}
		x=\sum_{n=1}^\infty (f_nS_{f, \tau}^{-1})(x) \tau_n=\sum_{n=1}^\infty
		f_n(x) S_{f, \tau}^{-1}\tau_n, \quad \forall x \in
		\mathcal{X}.
		\end{align*}
		\item The map 
		$
		\theta_f: \mathcal{X} \ni x \mapsto \{f_n(x) \}_n \in \ell^p(\mathbb{N})
		$
		is injective. 
		\item 
		The map
		$
		\theta_\tau: \ell^p(\mathbb{N}) \ni \{a_n \}_n \mapsto \sum_{n=1}^\infty a_n\tau_n \in \mathcal{X}
		$
		is surjective.
		\item  $S_{f, \tau}=\theta_\tau\theta_f.$
		\item  $P_{f, \tau}\coloneqq\theta_fS_{f,\tau}^{-1}\theta_\tau:\ell^p(\mathbb{N})\to \ell^p(\mathbb{N})$ is a projection onto   $\theta_f(\mathcal{X})$.
	\end{enumerate}
\end{theorem}

In order to derive the  dilation result we must have a notion of Riesz basis for Banach space. There are various characterizations for Riesz bases for Hilbert spaces (see Theorem 5.5.4 in \cite{OLEBOOK}, Theorem 7.13 in \cite{HEIL}, and a recent generalization by Stoeva  in \cite{STOEVA}) but all uses (implicitly or explicitly) inner product structures and orthonormal bases. These characterizations lead to the   notion of p-Riesz basis for Banach spaces using a single sequence in the Banach space (see \cite{ALDROUBI, CHRISTENSENSTOEVA}) but we do not consider that in this paper. \\
To define the notion of Riesz basis, which is compatible with Hilbert space situation, we first derive an operator-theoretic  characterization for Riesz basis in Hilbert spaces, which does not use the inner product of Hilbert space. To do so, we need a result from Hilbert space frame theory. 
\begin{theorem}\cite{HOLUB}\label{HOLUBTHEOREM} (Holub's theorem)
	A sequence $\{\tau_n\}_n$ in  $\mathcal{H}$ is a
	frame for $\mathcal{H}$	if and only if there exists a surjective bounded linear operator $T:\ell^2(\mathbb{N}) \to \mathcal{H}$ such that $Te_n=\tau_n$, for all $n \in \mathbb{N}$, where $\{e_n\}_n$ is the standard orthonormal basis for $\ell^2(\mathbb{N})$.
\end{theorem}
In the sequel, given a space $\mathcal{X}$, by $I_\mathcal{X}$ we mean the identity mapping on $\mathcal{X}$.
\begin{theorem}\label{RIESZBASISCHAROURS}
For  sequence $\{\tau_n\}_n$ in  $\mathcal{H}$, the following are equivalent.
\begin{enumerate}[\upshape (i)]
	\item $\{\tau_n\}_n$ is a Riesz basis for  $ \mathcal{H}$. 
	\item $\{\tau_n\}_n$ is a frame  for  $ \mathcal{H}$ and 
	\begin{align}\label{RIESZEQUATIONTHEOREM}
 \theta_\tau S_\tau^{-1} \theta_\tau^*=I_{\ell^2(\mathbb{N})}.
\end{align}
\end{enumerate}	
\end{theorem}
\begin{proof}
\begin{enumerate}[\upshape (i)]
	\item $\implies$ (ii) It is well-known that a Riesz basis is a frame (for instance, see Proposition 3.3.5 in \cite{OLEBOOK}).  Now there exist an orthonormal basis $\{\omega_n\}_n$ for  $\mathcal{H}$ and   bounded invertible operator $T: \mathcal{H}\to \mathcal{H}$ such that $T\omega_n=\tau_n$, for all $n \in \mathbb{N}$. We then have 
	\begin{align*}
	S_\tau h&= \sum_{n=1}^\infty \langle h, \tau_n\rangle\tau_n= \sum_{n=1}^\infty \langle h, T\omega_n\rangle T \omega_n\\
	&=T\left(\sum_{n=1}^\infty \langle T^*h, \omega_n\rangle  \omega_n\right)=TT^*h, \quad \forall h \in \mathcal{H}.
	\end{align*}
	Therefore 
	\begin{align*}
	\theta_\tau S_\tau^{-1} \theta_\tau^*\{a_n\}_n&=\theta_\tau (TT^*)^{-1} \theta_\tau^*\{a_n\}_n=\theta_\tau (T^*)^{-1}T^{-1} \theta_\tau^*\{a_n\}_n\\
	&=\theta_\tau (T^*)^{-1}T^{-1}\left(\sum_{n=1}^\infty a_n\tau_n\right)=\theta_\tau (T^*)^{-1}T^{-1}\left(\sum_{n=1}^\infty a_nT\omega_n\right)\\
	&=\theta_\tau \left(\sum_{n=1}^\infty a_n(T^*)^{-1}\omega_n\right)=\sum_{k=1}^{\infty}\left\langle \sum_{n=1}^\infty a_n(T^*)^{-1}\omega_n, \tau_k\right\rangle e_k \\
	&=\sum_{k=1}^{\infty}\left\langle \sum_{n=1}^\infty a_n(T^*)^{-1}\omega_n, T\omega_k\right\rangle e_k\\
	&=\sum_{k=1}^{\infty}\left\langle \sum_{n=1}^\infty a_n\omega_n, \omega_k\right\rangle e_k=\{a_k\}_k, \quad\forall\{a_n\}_n \in  \ell^2(\mathbb{N}).
	\end{align*}
\item $\implies$ (i) From Holub's theorem  (Theorem \ref{HOLUBTHEOREM}), there exists a surjective bounded linear operator $T:\ell^2(\mathbb{N}) \to \mathcal{H}$ such that $Te_n=\tau_n$, for all $n \in \mathbb{N}$. Since all separable Hilbert spaces are isometrically isomorphic to one another and orthonormal bases map into orthonormal bases, without loss of generality we may assume that $\{e_n\}_n$ is an orthonormal basis for $\mathcal{H}$ and the domain of $T$ is $\mathcal{H}$. Our job now reduces in showing $T$ is invertible. Since $T$ is already surjective, to show it is invertible, it suffices to show it is injective. Let $\{a_n\}_n \in  \ell^2(\mathbb{N}).$ Then $\{a_n\}_n=\theta_\tau (S_\tau^{-1} \theta_\tau^*\{a_n\}_n)$. Hence $\theta_\tau$ is surjective. We now find 
\begin{align*}
\theta_\tau h=\sum_{n=1}^{\infty}\langle h, \tau_n\rangle e_n=\sum_{n=1}^{\infty}\langle h, Te_n\rangle e_n=T^*h, \quad \forall h \in \mathcal{H}.
\end{align*}
Therefore 
	\begin{align*}
\operatorname{Kernel} (T)=T^*(\mathcal{H})^\perp=\theta_\tau(\mathcal{H})^\perp=\mathcal{H}^\perp=\{0\}.
	\end{align*}
Hence $T$ is injective.	
\end{enumerate}		
\end{proof}
Theorem \ref{RIESZBASISCHAROURS} leads to the following definition of p-approximate Riesz basis.
\begin{definition}
A pair $ (\{f_n \}_{n}, \{\tau_n \}_{n}) $ is said to be a p-approximate Riesz basis   for $\mathcal{X}$ if 	it is a p-ASF for $ \mathcal{X}$ and $\theta_fS_{f,\tau}^{-1}\theta_\tau=I_{\ell^p(\mathbb{N})}$.
\end{definition}
We now derive the dilation theorem. 
\begin{theorem}\label{DILATIONTHEOREMPASF}
(Dilation theorem)	Let   $ (\{f_n \}_{n}, \{\tau_n \}_{n}) $ be a p-ASF 	for  $\mathcal{X}$. Then there exist a Banach space $\mathcal{X}_1$	which contains $\mathcal{X}$ isometrically and a p-approximate Riesz basis  $ (\{g_n \}_{n}, \{\omega_n \}_{n}) $   	for  $\mathcal{X}_1$ such that 
	\begin{align*}
	f_n=g_nP_{|\mathcal{X}}, \quad\tau_n=P\omega_n, \quad \forall n \in \mathbb{N},
	\end{align*}
 where $P:\mathcal{X}_1\rightarrow \mathcal{X}$ is onto  projection. 
\end{theorem}
\begin{proof}
Let  $\{e_n\}_n$ denote the standard Schauder basis for  $\ell^p(\mathbb{N})$ and let $\{\zeta_n\}_n$ denote the coordinate functionals associated with $\{e_n\}_n$. 	Define
\begin{align*}
\mathcal{X}_1\coloneqq\mathcal{X}\oplus(I_{\ell^p(\mathbb{N})}-P_{f, \tau})(\ell^p(\mathbb{N})), \quad P:\mathcal{X}_1 \ni x\oplus y\mapsto x\oplus 0 \in \mathcal{X}_1
\end{align*}
  and 
\begin{align*}
\omega_n\coloneqq	\tau_n\oplus (I_{\ell^p(\mathbb{N})}-P_{f, \tau})e_n \in \mathcal{X}_1, \quad \quad g_n\coloneqq f_n \oplus \zeta_n (I_{\ell^p(\mathbb{N})}-P_{f, \tau})\in \mathcal{X}_1^*, \quad \forall  n \in \mathbb{N}.
\end{align*}
Then clearly $\mathcal{X}_1$	 contains $\mathcal{X}$ isometrically, $P:\mathcal{X}_1\rightarrow \mathcal{X}$ is onto projection and 
\begin{align*}
	&(g_nP_{|\mathcal{X}})(x)=g_n(P_{|\mathcal{X}}x)=g_n(x)=(f_n \oplus \zeta_n (I_{\ell^p(\mathbb{N})}-P_{f, \tau}))(x\oplus 0)=f_n(x),\quad \forall x \in \mathcal{X},\\
	& P\omega_n=P(\tau_n\oplus (I_{\ell^p(\mathbb{N})}-P_{f, \tau})e_n)=\tau_n, \quad \forall n \in \mathbb{N}.
\end{align*}
 Since the operator $I_{\ell^p(\mathbb{N})}-P_{f, \tau}$ is idempotent, it follows that $(I_{\ell^p(\mathbb{N})}-P_{f, \tau})(\ell^p(\mathbb{N}))$ is a closed subspace of $\ell^p(\mathbb{N})$ and hence a Banach space. Therefore $\mathcal{X}_1$ is a Banach space.  Let $x\oplus y \in \mathcal{X}_1$ and we shall write  $y=\{a_n\}_n \in \ell^p(\mathbb{N})$. We then see that 

\begin{align*}
&\sum_{n=1}^{\infty}(\zeta_n (I_{\ell^p(\mathbb{N})}-P_{f, \tau}))(y)\tau_n=\sum_{n=1}^{\infty}\zeta_n(y)\tau_n-\sum_{n=1}^{\infty}\zeta_n(P_{f, \tau}(y))\tau_n\\
&=\sum_{n=1}^{\infty}\zeta_n(\{a_k\}_{k})\tau_n-\sum_{n=1}^{\infty}\zeta_n(\theta_fS_{f, \tau}^{-1}\theta_\tau(\{a_k\}_{k}))\tau_n\\
&=\sum_{n=1}^{\infty}a_n\tau_n-\sum_{n=1}^{\infty}\zeta_n\left(\theta_fS_{f, \tau}^{-1}\left(\sum_{k=1}^{\infty}a_k\tau_k\right)\right)\tau_n
=\sum_{n=1}^{\infty}a_n\tau_n-\sum_{n=1}^{\infty}\zeta_n\left(\sum_{k=1}^{\infty}a_k\theta_fS_{f, \tau}^{-1}\tau_k\right)\tau_n\\
&=\sum_{n=1}^{\infty}a_n\tau_n-\sum_{n=1}^{\infty}\zeta_n\left(\sum_{k=1}^{\infty}a_k\sum_{r=1}^{\infty}f_r(S_{f, \tau}^{-1}\tau_k)e_r\right)\tau_n
=\sum_{n=1}^{\infty}a_n\tau_n-\sum_{n=1}^{\infty}\sum_{k=1}^{\infty}a_k\sum_{r=1}^{\infty}f_r(S_{f, \tau}^{-1}\tau_k)\zeta_n(e_r)\tau_n\\
&=\sum_{n=1}^{\infty}a_n\tau_n-\sum_{n=1}^{\infty}\sum_{k=1}^{\infty}a_kf_n(S_{f, \tau}^{-1}\tau_k)\tau_n=\sum_{n=1}^{\infty}a_n\tau_n-\sum_{k=1}^{\infty}a_k\sum_{n=1}^{\infty}f_n(S_{f, \tau}^{-1}\tau_k)\tau_n\\
&=\sum_{n=1}^{\infty}a_n\tau_n-\sum_{k=1}^{\infty}a_k\tau_k=0  \quad \text{ and}
\end{align*}
\begin{align*}
& \sum_{n=1}^{\infty}f_n(x)(I_{\ell^p(\mathbb{N})}-P_{f, \tau})e_n=\sum_{n=1}^{\infty}f_n(x)e_n-\sum_{n=1}^{\infty}f_n(x)P_{f, \tau}e_n\\
&=\sum_{n=1}^{\infty}f_n(x)e_n-\sum_{n=1}^{\infty}f_n(x)\theta_fS_{f, \tau}^{-1}\theta_\tau e_n=\sum_{n=1}^{\infty}f_n(x)e_n-\sum_{n=1}^{\infty}f_n(x)\theta_fS_{f, \tau}^{-1}\tau_n\\
&=\sum_{n=1}^{\infty}f_n(x)e_n-\sum_{n=1}^{\infty}f_n(x)\sum_{k=1}^{\infty}f_k(S_{f, \tau}^{-1}\tau_n)e_k
=\sum_{n=1}^{\infty}f_n(x)e_n-\sum_{n=1}^{\infty}\sum_{k=1}^{\infty}f_n(x)f_k(S_{f, \tau}^{-1}\tau_n)e_k\\
&=\sum_{n=1}^{\infty}f_n(x)e_n-\sum_{k=1}^{\infty}\sum_{n=1}^{\infty}f_n(x)f_k(S_{f, \tau}^{-1}\tau_n)e_k=\sum_{n=1}^{\infty}f_n(x)e_n-\sum_{k=1}^{\infty}f_k\left(\sum_{n=1}^{n}f_n(x)S_{f, \tau}^{-1}\tau_n\right)e_k\\
&=\sum_{n=1}^{\infty}f_n(x)e_n-\sum_{k=1}^{\infty}f_k(x)e_k=0.
\end{align*}
By using previous two calculations, we  get 
\begin{align*}
S_{g, \omega}(x\oplus y) &=\sum_{n=1}^{\infty}g_n(x\oplus y)\omega_n=\sum_{n=1}^{\infty}(f_n \oplus \zeta_n (I_{\ell^p(\mathbb{N})}-P_{f, \tau}))(x\oplus y)(\tau_n\oplus (I_{\ell^p(\mathbb{N})}-P_{f, \tau})e_n)\\
&=\sum_{n=1}^{\infty}(f_n(x) + (\zeta_n (I_{\ell^p(\mathbb{N})}-P_{f, \tau}))(y))(\tau_n\oplus (I_{\ell^p(\mathbb{N})}-P_{f, \tau})e_n)\\
&=\left(\sum_{n=1}^{\infty}f_n(x)\tau_n+\sum_{n=1}^{\infty}(\zeta_n (I_{\ell^p(\mathbb{N})}-P_{f, \tau}))(y)\tau_n\right)\oplus\\
&\quad \left(\sum_{n=1}^{\infty}f_n(x)(I_{\ell^p(\mathbb{N})}-P_{f, \tau})e_n+\sum_{n=1}^{\infty}(\zeta_n (I_{\ell^p(\mathbb{N})}-P_{f, \tau}))(y)(I_{\ell^p(\mathbb{N})}-P_{f, \tau})e_n\right)\\
&=(S_{f, \tau}x+0)\oplus \left(0+(I_{\ell^p(\mathbb{N})}-P_{f, \tau})\sum_{n=1}^{\infty}\zeta_n ((I_{\ell^p(\mathbb{N})}-P_{f, \tau})y)e_n\right)\\
&=S_{f, \tau}x\oplus (I_{\ell^p(\mathbb{N})}-P_{f, \tau})(I_{\ell^p(\mathbb{N})}-P_{f, \tau})y=S_{f, \tau}x\oplus (I_{\ell^p(\mathbb{N})}-P_{f, \tau})y\\
&=(S_{f, \tau}\oplus (I_{\ell^p(\mathbb{N})}-P_{f, \tau}))(x\oplus y).
\end{align*}
Since the operator $I_{\ell^p(\mathbb{N})}-P_{f, \tau}$ is idempotent, $I_{\ell^p(\mathbb{N})}-P_{f, \tau}$ becomes identity operator on the space $(I_{\ell^p(\mathbb{N})}-P_{f, \tau})(\ell^p(\mathbb{N}))$.  Hence we get that the operator  $S_{g, \omega}=S_{f, \tau}\oplus (I_{\ell^p(\mathbb{N})}-P_{f, \tau})$ is bounded invertible from  $\mathcal{X}_1$ onto itself. We next show that $ (\{g_n \}_{n}, \{\omega_n \}_{n}) $ is a p-approximate Riesz basis for $\mathcal{X}_1$. For this, first we find $\theta_g$ and $\theta_\omega$. Consider
\begin{align*}
\theta_g(x\oplus y)&=\{g_n(x\oplus y)\}_{n}=\{(f_n \oplus \zeta_n (I_{\ell^p(\mathbb{N})}-P_{f, \tau}))(x\oplus y)\}_{n}\\
&=\{f_n (x) +\zeta_n ((I_{\ell^p(\mathbb{N})}-P_{f, \tau}) y)\}_{n}=\{f_n (x)\}_{n} +\{\zeta_n ((I_{\ell^p(\mathbb{N})}-P_{f, \tau}) y)\}_{n}\\
&=\theta_fx+\sum_{n=1}^{\infty}\zeta_n ((I_{\ell^p(\mathbb{N})}-P_{f, \tau}) y )e_n=\theta_fx+(I_{\ell^p(\mathbb{N})}-P_{f, \tau}) y , \quad\forall x\oplus y \in \mathcal{X}_1
\end{align*}
and 
\begin{align*}
\theta_\omega\{a_n\}_n&=\sum_{n=1}^{\infty}a_n\omega_n=\sum_{n=1}^{\infty}a_n(\tau_n\oplus (I_{\ell^p(\mathbb{N})}-P_{f, \tau})e_n)\\
&= \left(\sum_{n=1}^{\infty}a_n\tau_n\right) \oplus \left(\sum_{n=1}^{\infty}a_n(I_{\ell^p(\mathbb{N})}-P_{f, \tau})e_n\right)\\
&=\theta_\tau\{a_n\}_n\oplus (I_{\ell^p(\mathbb{N})}-P_{f, \tau})\left(\sum_{n=1}^{\infty}a_ne_n\right)\\
&=\theta_\tau\{a_n\}_n\oplus (I_{\ell^p(\mathbb{N})}-P_{f, \tau})\{a_n\}_n, \quad\forall \{a_n\}_n \in \ell^p(\mathbb{N}).
\end{align*}
Therefore
\begin{align*}
P_{g, \omega}\{a_n\}_n&=\theta_g S_{g, \omega}^{-1}\theta_\omega\{a_n\}_n=\theta_gS_{g, \omega}^{-1}(\theta_\tau\{a_n\}_n\oplus (I_{\ell^p(\mathbb{N})}-P_{f, \tau})\{a_n\}_n)\\
&=\theta_g(S_{f, \tau}^{-1}\oplus (I_{\ell^p(\mathbb{N})}-P_{f, \tau}) )(\theta_\tau\{a_n\}_n\oplus (I_{\ell^p(\mathbb{N})}-P_{f, \tau})\{a_n\}_n)\\
&=\theta_g(S_{f, \tau}^{-1} \theta_\tau\{a_n\}_n\oplus  (I_{\ell^p(\mathbb{N})}-P_{f, \tau})^2\{a_n\}_n)\\
&=\theta_g(S_{f, \tau}^{-1} \theta_\tau\{a_n\}_n\oplus  (I_{\ell^p(\mathbb{N})}-P_{f, \tau})\{a_n\}_n)\\
&=\theta_f(S_{f, \tau}^{-1} \theta_\tau\{a_n\}_n)+(I_{\ell^p(\mathbb{N})}-P_{f, \tau})(I_{\ell^p(\mathbb{N})}-P_{f, \tau})\{a_n\}_n\\
&=P_{f, \tau}\{a_n\}_n+(I_{\ell^p(\mathbb{N})}-P_{f, \tau})\{a_n\}_n=\{a_n\}_n, \quad \forall \{a_n\}_n \in \ell^p(\mathbb{N}).
\end{align*}	
\end{proof}
\begin{corollary}
Theorem \ref{NHLDILATIONTHEOREM} is a corollary of Theorem \ref{DILATIONTHEOREMPASF}.	
\end{corollary}
\begin{proof}
Let $\{\tau_n\}_n$ be a frame for  $\mathcal{H}$. Define
\begin{align*}
f_n:\mathcal{H} \ni h \mapsto f_n(h)\coloneqq \langle h, \tau_n\rangle \in \mathbb{K}, \quad \forall n \in \mathbb{N}.
\end{align*}	
Then $\theta_f=\theta_\tau$. Note that now $ (\{f_n \}_{n}, \{\tau_n \}_{n}) $ is  a 2-approximate frame   for $\mathcal{H}$. Theorem  \ref{DILATIONTHEOREMPASF} now says that  there exist a Banach space $\mathcal{X}_1$	which contains $\mathcal{H}$ isometrically and a 2-approximate Riesz basis  $ (\{g_n \}_{n}, \{\omega_n \}_{n}) $   	for  $\mathcal{X}_1=\mathcal{H}\oplus(I_{\ell^2(\mathbb{N})}-P_{\tau})(\ell^2(\mathbb{N}))$ such that 
\begin{align*}
f_n=g_nP_{|\mathcal{H}}, \quad\tau_n=P\omega_n, \quad \forall n \in \mathbb{N},
\end{align*}
where $P:\mathcal{X}_1\rightarrow \mathcal{H}$ is onto  projection. Since   $(I_{\ell^2(\mathbb{N})}-P_{\tau})(\ell^2(\mathbb{N}))$ is a closed  subspace of the Hilbert space $\ell^2(\mathbb{N})$, $\mathcal{X}_1$ now becomes a Hilbert space. From the definition of $P$ we get that it is an  orthogonal projection. Now to prove Theorem \ref{NHLDILATIONTHEOREM}, we are left with proving $\{\omega_n\}_n$ is a Riesz basis for  $\mathcal{X}_1$. To show $\{\omega_n\}_n$ is a Riesz basis for  $\mathcal{X}_1$, we use  Theorem \ref{RIESZBASISCHAROURS}. Since $\{\tau_n\}_n$  is  a
frame for $\mathcal{H}$  there exist $a,b>0$ such that 
\begin{align*}
	a\|h\|^2 \leq \sum_{n=1}^\infty |\langle h, \tau_n\rangle|^2\leq b\|h\|^2, \quad \forall h \in \mathcal{H}.
\end{align*}
Let $h\oplus  (I_{\ell^2(\mathbb{N})}-P_{f, \tau})\{a_k\}_k\in \mathcal{X}_1$. Then by noting $b\geq1$, we get 
\begin{align*}
\sum_{n=1}^{\infty}|\langle h\oplus (I_{\ell^2(\mathbb{N})}-P_{ \tau})\{a_k\}_k, \omega_n\rangle|^2&=\sum_{n=1}^{\infty}|\langle h\oplus (I_{\ell^2(\mathbb{N})}-P_{ \tau})\{a_k\}_k, \tau_n\oplus (I_{\ell^2(\mathbb{N})}-P_{ \tau})e_n\rangle|^2\\
&=\sum_{n=1}^{\infty}|\langle h, \tau_n\rangle|^2+\sum_{n=1}^{\infty}|\langle  (I_{\ell^2(\mathbb{N})}-P_{ \tau})\{a_k\}_k,  (I_{\ell^2(\mathbb{N})}-P_{ \tau})e_n\rangle|^2\\
&=\sum_{n=1}^{\infty}|\langle h, \tau_n\rangle|^2+\sum_{n=1}^{\infty}|\langle    (I_{\ell^2(\mathbb{N})}-P_{ \tau})(I_{\ell^2(\mathbb{N})}-P_{ \tau})\{a_k\}_k, e_n\rangle|^2\\
&=\sum_{n=1}^{\infty}|\langle h, \tau_n\rangle|^2+\sum_{n=1}^{\infty}|\langle    (I_{\ell^2(\mathbb{N})}-P_{ \tau})\{a_k\}_k, e_n\rangle|^2\\
&=\sum_{n=1}^{\infty}|\langle h, \tau_n\rangle|^2+\|(I_{\ell^2(\mathbb{N})}-P_{ \tau})\{a_k\}_k\|^2\\
&\leq b\|h\|^2+\|(I_{\ell^2(\mathbb{N})}-P_{\tau})\{a_k\}_k\|^2\\
&\leq b(\|h\|^2+\|(I_{\ell^2(\mathbb{N})}-P_{ \tau})\{a_k\}_k\|^2)\\
&=b\| h\oplus (I_{\ell^2(\mathbb{N})}-P_{\tau})\{a_k\}_k\|^2.
\end{align*}
Previous calculation tells that $\{\omega_n\}_n$  is  a Bessel sequence
 for $\mathcal{X}_1$. Hence $S_\omega:\mathcal{X}_1 \ni x\oplus\{a_k\}_k\mapsto \sum_{n=1}^{\infty} \langle x\oplus\{a_k\}_k, \omega_n\rangle \omega_n\in \mathcal{X}_1$ is a well-defined bounded linear operator. Next we claim that 
 \begin{align}\label{CLAIM}
 g_n(x\oplus\{a_k\}_k)=\langle x\oplus\{a_k\}_k, \omega_n\rangle, \quad \forall x+\{a_k\}_k \in \mathcal{X}_1, \forall n \in \mathbb{N}.
 \end{align}
 Consider 
 \begin{align*}
  g_n(x\oplus\{a_k\}_k)&=(f_n \oplus \zeta_n (I_{\ell^2(\mathbb{N})}-P_{ \tau}))(x\oplus\{a_k\}_k)\\
  &=f_n(x)+\zeta_n ((I_{\ell^2(\mathbb{N})}-P_{ \tau})\{a_k\}_k)
  =f_n(x)+\zeta_n\left(\{a_k\}_k\right)-\zeta_n (P_{ \tau}\{a_k\}_k)\\
  &=f_n(x)+\zeta_n\left(\{a_k\}_k\right)-\zeta_n (\theta_\tau S_{ \tau}^{-1}\theta_\tau^*\{a_k\}_k)
  =f_n(x)+a_n-\zeta_n\left(\theta_\tau S_{ \tau}^{-1}\left(\sum_{k=1}^{\infty}a_k\tau_k\right)\right)\\
  &=f_n(x)+a_n-\zeta_n\left(\sum_{k=1}^{\infty}a_k\theta_\tau S_{ \tau}^{-1}\tau_k\right)
  =f_n(x)+a_n-\zeta_n\left(\sum_{k=1}^{\infty}a_k\sum_{r=1}^{\infty}\langle S_{ \tau}^{-1}\tau_k, \tau_r \rangle e_r\right)\\
  &=f_n(x)+a_n-\sum_{k=1}^{\infty}a_k\langle S_{ \tau}^{-1}\tau_k, \tau_n \rangle= \langle x, \tau_n\rangle+a_n-\sum_{k=1}^{\infty}a_k\langle S_{ \tau}^{-1}\tau_k, \tau_n \rangle \quad \text{ and }
 \end{align*}
  \begin{align*}
 \langle x\oplus\{a_k\}_k, \omega_n\rangle&=\langle x\oplus\{a_k\}_k, \tau_n\oplus (I_{\ell^2(\mathbb{N})}-P_{ \tau})e_n\rangle\\
 &=\langle x, \tau_n\rangle+\langle \{a_k\}_k,  (I_{\ell^2(\mathbb{N})}-P_{ \tau})e_n\rangle
 =\langle x, \tau_n\rangle+\langle \{a_k\}_k,  e_n\rangle+\langle \{a_k\}_k,  P_{ \tau}e_n\rangle\\
 &=\langle x, \tau_n\rangle+a_n-\left \langle\{a_k\}_k, \theta_\tau S_{ \tau}^{-1}\theta_\tau^*e_n\right\rangle=\langle x, \tau_n\rangle+a_n-\left \langle\{a_k\}_k, \theta_\tau S_{ \tau}^{-1}\tau_n\right\rangle\\
 &=\langle x, \tau_n\rangle+a_n- \langle\{a_k\}_k,  \{ \langle S_{ \tau}^{-1}\tau_n, \tau_k \rangle\}_k \rangle=\langle x, \tau_n\rangle+a_n-\sum_{k=1}^{\infty}a_k\overline{\langle S_{ \tau}^{-1}\tau_n, \tau_k \rangle}\\
 &=\langle x, \tau_n\rangle+a_n-\sum_{k=1}^{\infty}a_k\langle\tau_k,S_{ \tau}^{-1}\tau_n \rangle=\langle x, \tau_n\rangle+a_n-\sum_{k=1}^{\infty}a_k\langle S_{ \tau}^{-1}\tau_k, \tau_n \rangle.
 \end{align*}
 Thus Equation (\ref{CLAIM}) holds. Therefore for all $x\oplus\{a_k\}_k\in \mathcal{X}_1$,
 \begin{align*}
 S_{g,\omega}(x\oplus\{a_k\}_k)=\sum_{n=1}^{\infty}g_n(x\oplus\{a_k\}_k)\omega_n=\sum_{n=1}^{\infty}\langle x\oplus\{a_k\}_k, \omega_n\rangle\omega_n=S_{\omega}(x\oplus\{a_k\}_k).
 \end{align*}
 Since $ S_{g,\omega}$ is invertible, $S_{\omega}$ becomes invertible. Clearly $S_{\omega}$ is positive. Therefore 
 \begin{align*}
 \frac{1}{\|S_{\omega}\|^{-1}}\|g\|^2\leq \langle S_{\omega}g, g\rangle\leq \|S_\omega\| \|g\|^2, \quad \forall g \in \mathcal{X}_1. 
 \end{align*}
 Hence
 \begin{align*}
 \frac{1}{\|S_{\omega}\|^{-1}}\|g\|^2\leq \sum_{n=1}^{\infty}|\langle g, \omega_n\rangle|^2\leq \|S_\omega\| \|g\|^2, \quad \forall g \in \mathcal{X}_1. 
 \end{align*}
 That is,  $\{\omega_n\}_n$  is  a frame 
 for $\mathcal{X}_1$.
 
 Finally we  show Equation (\ref{RIESZEQUATIONTHEOREM}) in  Theorem \ref{RIESZBASISCHAROURS} for the frame $\{\omega_n\}_n$. Consider 
 \begin{align*}
 \theta_\omega S_\omega^{-1} \theta_\omega^*\{a_n\}_n&=\theta_\omega S_\omega^{-1}\left(\sum_{n=1}^{\infty}a_n\omega_n\right)=\theta_\omega \left(\sum_{n=1}^{\infty}a_nS_\omega^{-1}\omega_n\right)\\
 &=\sum_{k=1}^{\infty} \left\langle \sum_{n=1}^{\infty}a_nS_\omega^{-1}\omega_n, \omega_k\right \rangle=\sum_{k=1}^{\infty}  \sum_{n=1}^{\infty}a_n\langle S_\omega^{-1}\omega_n, \omega_k \rangle\\
 &=\sum_{k=1}^{\infty}  \sum_{n=1}^{\infty}a_n\langle (S_\tau^{-1}\oplus(I_{\ell^2(\mathbb{N})}-P_{ \tau}) )(\tau_n\oplus (I_{\ell^2(\mathbb{N})}-P_{ \tau})e_n), \tau_k\oplus (I_{\ell^2(\mathbb{N})}-P_{ \tau})e_k \rangle\\
 &=\sum_{k=1}^{\infty}  \sum_{n=1}^{\infty}a_n\langle (S_\tau^{-1}\tau_n\oplus(I_{\ell^2(\mathbb{N})}-P_{ \tau}) ^2e_n, \tau_k\oplus (I_{\ell^2(\mathbb{N})}-P_{ \tau})e_k \rangle\\
 &=\sum_{k=1}^{\infty}  \sum_{n=1}^{\infty}a_n(\langle S_\tau^{-1}\tau_n, \tau_k \rangle+\langle (I_{\ell^2(\mathbb{N})}-P_{ \tau}) e_n,  (I_{\ell^2(\mathbb{N})}-P_{ \tau})e_k \rangle)\\
 &=\sum_{k=1}^{\infty} \left\langle \sum_{n=1}^{\infty}a_nS_\tau^{-1}\tau_n, \tau_k\right \rangle+\sum_{k=1}^{\infty}  \sum_{n=1}^{\infty}a_n\langle (I_{\ell^2(\mathbb{N})}-P_{f, \tau}) e_n,  (I_{\ell^2(\mathbb{N})}-P_{ \tau})e_k \rangle\\
 &=P_\tau\{a_n\}_n+\sum_{k=1}^{\infty}  \sum_{n=1}^{\infty}a_n\langle (I_{\ell^2(\mathbb{N})}-P_{ \tau}) e_n,e_k \rangle\\
 &=P_\tau\{a_n\}_n+\sum_{k=1}^{\infty}\sum_{n=1}^{\infty}a_n\langle  e_n,e_k \rangle-\sum_{k=1}^{\infty}\sum_{n=1}^{\infty}a_n\langle  P_\tau e_n,e_k \rangle\\
 &=P_\tau\{a_n\}_n+\sum_{k=1}^{\infty}a_ke_k-\sum_{k=1}^{\infty}\sum_{n=1}^{\infty}a_n\langle  \theta_\tau S^{-1}_\tau  \theta_\tau ^*e_n,e_k \rangle\\
 &=P_\tau\{a_n\}_n+\sum_{k=1}^{\infty}a_ke_k-\sum_{k=1}^{\infty}\sum_{n=1}^{\infty}a_n\langle   S^{-1}_\tau  \tau_n,\theta_\tau^*e_k \rangle\\
 &=P_\tau\{a_n\}_n+\sum_{k=1}^{\infty}a_ke_k-\sum_{k=1}^{\infty}\sum_{n=1}^{\infty}a_n\langle   S^{-1}_\tau  \tau_n,\tau_k \rangle\\
 &=P_\tau\{a_n\}_n+\sum_{k=1}^{\infty}a_ke_k-P_\tau\{a_n\}_n= \{a_n\}_n,\quad\forall  \{a_n\}_n \in \ell^2(\mathbb{N}).
 \end{align*}
 Thus  $\{\omega_n\}_n$  is  a Riesz basis 
 for $\mathcal{X}_1$  which completes the proof.
\end{proof}

  \section{Acknowledgements}
 First author thanks National Institute of Technology Karnataka (NITK) Surathkal for financial assistance.

 \bibliographystyle{plain}
 \bibliography{reference.bib}

\end{document}